\documentclass[11pt, reqno]{amsart}

\usepackage{amssymb,latexsym,amsmath,amsfonts}
\usepackage{mathrsfs}
\usepackage{graphicx}
\usepackage[usenames]{color}

\numberwithin{equation}{section}

\allowdisplaybreaks

\theoremstyle{definition}
\newtheorem{definition}{Definition}[section]

\theoremstyle{remark}
\newtheorem{remark}[definition]{Remark}

 \theoremstyle{plain}
\newtheorem{theorem}[definition]{Theorem}
\newtheorem{result}[definition]{Result}
\newtheorem{lemma}[definition]{Lemma}

\newtheorem{example}[definition]{Example}

\newcommand{\lam}{\lambda}

\newcommand{\zt}{\zeta}

\newcommand{\A}{\mathsf{A}}
\newcommand{\B}{\mathsf{B}}

\newcommand{\bdy}{\partial}
\newcommand{\OM}{\Omega}
\newcommand{\D}{\mathbb{D}}

\newcommand{\hol}{\mathcal{O}}

\newcommand\hyper[2]{\left|\frac{{#1}-{#2}}{1-\overline{{#2}}{#1}}\right|}
\newcommand{\mobi}{\mathcal{M}_{\mathbb{D}}}
\newcommand{\Z}{\mathbb{Z}}
\newcommand{\nat}{\mathbb{N}}
\newcommand{\bcdot}{\boldsymbol{\cdot}}
\newcommand\intgR[2]{[{#1}\,.\,.\,{#2}]}
\newcommand{\KDel}{\boldsymbol{\delta}}
\newcommand\minpo[1]{\boldsymbol{{\sf M}}_{{#1}}}
\newcommand{\facto}[2]{\genfrac{(}{)}{0pt}{0}{#1}{#2}}
\newcommand\intf[4]{\genfrac{#1}{#2}{0.5pt}{0}{#3}{#4}}
\newcommand{\lrarw}{\longrightarrow}
\newcommand\bv[1]{{#1}^\bullet}
\newcommand\sub[2]{\underset{#1}{#2}}
\newcommand{\hausd}{\mathcal{H}^{C_{\Omega}}}
\newcommand{\multF}{F_{\raisebox{-3pt}{$\scriptstyle{\!\Gamma}$}}}

\newcommand{\C}{\mathbb{C}}

\makeatletter
\newcommand*{\rom}[1]{\expandafter\@slowromancap\romannumeral #1@}
\makeatother

\begin{document}

\title[Spectral Pick interpolation \& holomorphic correspondences]{the 3-point spectral Pick
interpolation problem and \\ an application to holomorphic correspondences}

\author{Vikramjeet Singh Chandel}
\address{Department of Mathematics, Indian Institute of Science, Bangalore 560012, India}
\email{vikramjeetc@iisc.ac.in}


\keywords{spectral unit ball, minimal Blaschke product, holomorphic functional calculus, holomorphic correspondences}
\subjclass[2010]{Primary: 30E05, 32H35, 47A56; Secondary: 32F45}

\begin{abstract}
We provide a necessary condition for the existence of a $3$-point holomorphic 
interpolant $F:\D\lrarw\OM_n$, $n\geq 2$. Our condition is inequivalent to the necessary conditions
hitherto known for this problem.
The condition generically involves a single inequality and is reminiscent of the Schwarz lemma.
We combine some of the ideas and techniques used in our result on the $\hol(\D,\,\OM_n)$-interpolation
problem to establish a Schwarz lemma\,---\,which may be of independent interest\,---\,for
holomorphic correspondences from $\D$ to a general planar domain $\OM\Subset\C$.
\end{abstract}
\maketitle

\section{Introduction and statement of results}\label{S:intro}

Let $\D$ denote the open unit disc in the complex plane $\C$ centered at $0$. Given $n\in\Z_{+}$ the
$n^2$-dimensional {\em spectral unit ball} is the set $\OM_n:=\{A\in M_{n}(\C)\,:\,\sigma(A)\subset\D\}$,
where $M_n(\C)$ denotes the set of all $n\times n$ complex matrices and $\sigma$ denotes the
spectrum of a matrix. The interpolation problem referred to in the title of this article
is the following problem:

\begin{itemize}
\item[$(*)$] {Given $M$ distinct points $\zt_1,\dots,\zt_M\in \D$ and matrices $W_1,\dots, W_M\in\OM_n$,
$n\geq 2$, find conditions on the data $\{(\zt_j,\,W_j):1\leq j\leq M\}$ such that there
exists a holomorphic map $F:\D\longrightarrow\OM_n$ satisfying $F(\zt_j)=W_j, \ j=1,\dots,M$.}
\end{itemize}
When such a function $F$ exists, we shall say that $F$ is an {\em interpolant} of the data.
\smallskip

One of the important steps towards understanding the problem $(*)$ was an operator-theoretic
approach due to Bercovici, Foias and Tannenbaum. Using a spectral version of the commutant-lifting theorem,
the authors in \cite{HariCiprianAllen:asclt91}\,---\,under the restriction that
$\sup_{\zt\in\D}\rho(F(\zt))<1$, where $\rho$ denotes the spectral radius\,---\,provided
a characterization for the existence of an interpolant.
This characterization involves a search for $M$ appropriate matrices in $GL_n(\C)$.
\smallskip

Another influential idea was introduced by Agler and Young in \cite{aglerYoung:cldC2si99}. They observed that 
in the case $W_1,\dots,W_M$ are all non-derogatory, then $(*)$ is equivalent to an interpolation
problem from $\D$ to the $n$-dimensional {\em symmetrized polydisc} $G_n$, $n\geq 2$. This is a bounded 
domain in $\C^n$ (see \cite{costara:osNPp05} for the definition of $G_n$). Its relevance to $(*)$ is that, for ``generic'' matricial 
data $(W_1,\dots,W_M)$, the problem $(*)$ descends to a region of much lower dimension with many pleasant
properties. This idea has further been developed in \cite{aglerYoung:2psNPp00}, in the papers \cite{costara:22sNPp05} and
\cite{costara:osNPp05} by Costara, and in Ogle's thesis \cite{ogle:thesis99}.
A matrix $A\in M_n(\C)$ is said to be \emph{non-derogatory} if it admits a cyclic vector.
It is a fact that $A$ being non-derogatory is equivalent to $A$ being similar to the companion
matrix of its characteristic polynomial (see \cite[p.~195]{hornJohn:matanal85}, for instance). Recall:
given a monic polynomial of degree $k$ of the form $p(t)=t^k+\sum_{j=1}^ka_j\,t^{k-j}$, where $a_j\in\C$,
the \emph{companion matrix} of $p$ is the matrix $\mathsf{C}_p\in M_{k}(\C)$ given by
\[
\mathsf{C}_p:=
\begin{bmatrix}
\ 0  & {} & {} & -a_k \ \\
\ 1  & 0  & {} & -a_{k-1} \ \\
\ {} & \ddots  & \ddots & \vdots \ \\
\ \text{\LARGE{0}} &   & 1 & -a_{1} \
\end{bmatrix}_{k\times k}.
\]
By way of the $G_n$-interpolation
problem, Costara \cite{costara:osNPp05} and Ogle \cite{ogle:thesis99} arrived independently at
a necessary condition for the existence of an interpolant for the problem $(*)$ when the data 
$(W_1,\dots,W_M)$ are non-derogatory.
\smallskip

Bharali in \cite{gb:itplSpUb07} observed that when $n\geq 3$, the necessary condition given in
\cite{costara:osNPp05, ogle:thesis99} is not sufficient. He also established\,---\,for the case $M=2$\,---\,a
new necessary condition for the existence of
an interpolant. Result~\ref{Res:gbintS} below is this necessary condition.
It is reminiscent of the inequality in the classical Schwarz lemma; here $\mobi(z_1,z_2)$ is 
the {\em M{\"o}bius distance} between $z_1$ and $z_2$, which is defined as:
\[
 \mobi(z_1,z_2) \ := \ \hyper{z_1}{z_2} \quad\forall z_1,z_2\in\D.
\]

\begin{result}[Bharali, \cite{gb:itplSpUb07}]\label{Res:gbintS}
Let $F\in\hol(\D,\,\OM_n)$, $n\geq 2$, and let $\zt_1,\zt_2\in\D$. 
Write $W_j=F(\zt_j)$, and if $\lam\in\sigma(W_j)$, then let $m(\lam)$ denote the multiplicity of $\lam$ as a
zero of the minimal polynomial of $W_j$. Then:
\begin{equation}\label{E:SchwarzIneq}
\max\left\{\max_{\mu\in\sigma(W_2)}\prod_{\lam\in\sigma(W_1)}\mobi(\mu,\lam)^{m(\lam)}, 
\ \max_{\lambda\in\sigma(W_1)}\prod_{\mu\in\sigma(W_2)}\mobi(\lam,\mu)^{m(\mu)}\right\} \ 
\leq \ \hyper{\zt_1}{\zt_2}.
\end{equation}
\end{result}

The above theorem gives a necessary condition for the two-point interpolation problem 
without any restriction on the matrices, in contrast to the necessary condition
in \cite{costara:osNPp05, ogle:thesis99}. In the same article,
Bharali also shows that for each $n\geq 3$, there exists a data-set for which \eqref{E:SchwarzIneq} implies 
that it cannot admit an interpolant whereas the condition in \cite{costara:osNPp05, ogle:thesis99} is inconclusive.
\smallskip

The ideas behind Result~\ref{Res:gbintS} strongly influence a part of this work.
One of the key tools introduced in \cite{gb:itplSpUb07} that lead to Result~\ref{Res:gbintS} are the following
maps:

\begin{definition}
Given $A\in M_n(\C)$, let $\minpo{A}$ denote its minimal polynomial and write:
\[
\minpo{A}(t)=\sum_{\lam\in\sigma(A)}(t-\lam)^{m(\lam)}.
\]
The finite Blaschke product induced by $\minpo{A}$ if $A\in\OM_n$:
\begin{equation}
B_{A}(t):=\prod_{\lam\in\sigma(A)\subset\D}{\intf{(}{)}{t-\lam}{1-\overline{\lam}t}^{m(\lam)}}.
\label{E:mbpcmp}
\end{equation}
will be called the {\em minimal Blaschke product corresponding to $A$}.
\end{definition}

\noindent $B_{A}$ induces, via the holomorphic functional calculus (which we will discuss in
Section~\ref{S:holo_fc}), a holomorphic self-map of $\OM_n$ that maps $A$ to $0\in M_n(\C)$.
This sets up a form of the Schur algorithm on $\OM_n$, and yields
an easy-to-check necessary condition for the existence of an interpolant for the data in $(*)$,
for the case $M=3$. The existence of these maps $B_A$ is extremely useful, since 
the automorphism group of $\OM_n$ does not act transitively on $\OM_n$
(see \cite{ransfordWhite:hsmsub91}), $n\geq 2$ (whence the {\em classical} Schur algorithm is not
even meaningful).
\smallskip

In \cite{baribeauKamara:rSlsnpI14}, Baribeau and Kamara take a new look at the ideas in 
\cite{gb:itplSpUb07}. This they combine with an inequality\,---\,which may be viewed as a Schwarz lemma
for {\em algebroid multifunctions} of the unit disc (see \cite{nokraneRansford:Slam01} for a definition)\,---\,due
to Nokrane and Ransford \cite[Theorem~1.1]{nokraneRansford:Slam01}. Before we present
their result we need the following:
given $F\in\hol(\D,\,\OM_n)$ and $\zt_1\in\D$, if we denote by $B_1$ the minimal Blaschke product 
corresponding to $F(\zt_1)$ then Theorem~1.3 in \cite{baribeauKamara:rSlsnpI14} states, essentially, 
that for every $\zt\in\D$ we have
\[
 \sigma\big(B_1(F(\zt))/\psi_1(\zt)\big)= S\cup\sigma(F_1(\zt)),
\]
where $S\subset\bdy{\D}$ is a finite (possibly empty) set independent of $\zt$, $F_1\in\hol(\D,\,\OM_{\nu})$, and
\begin{equation}\label{E:nu_BK} 
\nu = \max\nolimits_{\zt\in \D}\left|\sigma\big(B_1(F(\zt))/\psi_1(\zt)\big)\cap\D\right|.
\end{equation}
Here and in what follows, for $\zt_j\in\D$, $j=1,2,3$, $\psi_j$ will denote
the automorphism $\psi_j(\zt):=(\zt-\zt_j)(1-\overline{\zt}_j\zt)^{-1}$, $\zt\in\D$, of $\D$.
We are now in a position to state:

 \begin{result}[paraphrasing {\cite[Corollary~3.1]{baribeauKamara:rSlsnpI14}}]\label{Res:inequality_baribeauKamara}
Let $\zt_1,\zt_2,\zt_3$ be distinct points in $\D$. Let $F\in\hol(\D,\,\OM_n)$, $n\geq 2$. Denote by $B_1$ the minimal 
Blaschke product corresponding to $F(\zt_1)$, and suppose that
$\sigma\big({B_1(F(\zt))}/{\psi_1(\zt)}\big)\not\subset\bdy{\D}$
for every $\zt\in\D$. Let $\nu$ be the number given by \eqref{E:nu_BK}.
Then we have:
\begin{equation}\label{E:inequality_baribeauKamara}
\mathcal{H}^{\mobi}\left(\sigma\intf{(}{)}{B_1(W_{2})}{\psi_1(\zt_{2})}\cap\D,\;\;
\sigma\intf{(}{)}{B_1(W_{3})}{\psi_1(\zt_{3})}\cap\D\right)
\leq{\mobi(\zt_{2},\,\zt_{3})}^{1/{\nu}}
\end{equation}
where $W_j=F(\zt_j)$, $j=2,\,3$.
\end{result}

\noindent{Here, $\mathcal{H}^{\mobi}$ denotes the Hausdorff distance induced by the M{\"o}bius distance
(see \cite[p.~279]{munk:topo_74} for the definition of Hausdorff distance) on the class of bounded
subsets of $\D$.}
\smallskip

Now we are ready to present the first result of this article (in what follows, $B_j$ will denote the 
minimal Blaschke product\,---\,as well as its extension to $\OM_n$\,---\,associated to the matrix $W_j$,
$j=1,2,3$):

\begin{theorem}\label{T:3pt_nec}
Let $\zt_1,\zt_2,\zt_3\in\D$ be distinct points and let $W_1,W_2,W_3\in\Omega_n$, $n\geq 2$. 
Let $m(j,\,\lambda)$ denote the multiplicity of $\lambda$ as a zero of the minimal polynomial
of $W_j$, $j\in\{1,2,3\}$.
Given $j,k\in\{1,2,3\}$ such that $j\not=k$, and $\nu\in \D$, we write:
\[
q(\nu,j,k):=\max\left\{\intf{[}{]}{m(j,\,\lambda)-1}
{\mathsf{ord}_{\lambda}{B'_k}+1}+1:\,\lambda\in \sigma(W_j)\cap B_k^{-1}\{\nu\} \right\}.
\]
Finally, for each $k\in\{1,2,3\}$ let
\[
G(k):=\max\,(\{1,2,3\}\setminus\{k\}),\,\,\text{and}\,\,\,
L(k):=\min\,(\{1,2,3\}\setminus\{k\}).
\]
If there exists a map $F\in\hol(\D,\,\Omega_n)$ such that $F(\zt_j)\,=\,W_j$, $j\in\{1,2,3\}$,
then for each $k\in\{1,2,3\}$, we have: 
\begin{itemize}
\item either $\sigma\left(B_k(W_{G(k)})\right)\subset
D\left(0,\,|\,\psi_{k}(\zt_{G(k)})\,|\right)$,
$\sigma\left(B_k(W_{L(k)})\right)\subset D\left(0,\,|\,\psi_{k}(\zt_{L(k)})\,|\right)$ and
\begin{align*}
{}&\max\left\{\sub{\mu\in\sigma\left(B_k(W_{L(k)})\right)}{\max}
\prod_{\nu\in\sigma\left(B_k(W_{G(k)})\right)}
{\mathcal{M}_{\D}\left(\intf{}{}{\mu}{\psi_k(\zt_{L(k)})},\,\intf{}{}{\nu}{\psi_k(\zt_{G(k)})}
\right)}^{q(\nu,\,G(k),\,k)},\right.\\
&\left.\sub{\mu\in\sigma\left(B_k(W_{G(k)})\right)}{\max}\prod_{\nu\in\sigma\left(B_k(W_{L(k)})\right)}
{\mathcal{M}_{\D}\left(\intf{}{}{\mu}{\psi_k(\zt_{G(k)})},\,\intf{}{}{\nu}{\psi_k(\zt_{L(k)})}
\right)}^{q(\nu,\,L(k),\,k)}\right\}\leq\mathcal{M}_{\D}\left(\zt_{L(k)},\,\zt_{G(k)}\right),
\end{align*}

\item or there exists a $\theta_0\in\mathbb{R}$ such that
\[
{B_k}^{-1}\{e^{i\theta_0}\psi_{k}(\zt_{G(k)})\}\subseteq\sigma(W_{G(k)}) \  \text{and} \
{B_k}^{-1}\{e^{i\theta_0}\psi_{k}(\zt_{L(k)})\}\subseteq\sigma(W_{L(k)}).
\]
\end{itemize}
\end{theorem}
\noindent Here, $[\bcdot]$ denotes the greatest-integer function. Given $a\in \C$ and a function $g$ that is
holomorphic in a neighbourhood of $a$, $\mathsf{ord}_a{g}$ will denote the order of vanishing of
$g$ at $a$ (with the understanding that $\mathsf{ord}_a{g} = 0$ if $g$ does not vanish at $a$).
\smallskip

\begin{remark}\label{Rmk:comparison}
Theorem~\ref{T:3pt_nec}, unlike Result \ref{Res:inequality_baribeauKamara}, incorporates information
about the Jordan structure of the matricial data. Thus, Theorem~\ref{T:3pt_nec} gives a more restrictive
inequality than \eqref{E:inequality_baribeauKamara} if $\nu=n$. Moreover, in Section~\ref{S:3pt_nec_proof}
we will present a class of $3$-point matricial data in $\D\times\OM_n$, $n\geq 4$,
for which the condition \eqref{E:inequality_baribeauKamara}
and that in \cite{costara:osNPp05, ogle:thesis99} provide no information while 
Theorem~\ref{T:3pt_nec} implies that these data do not admit a $\hol(\D,\,\OM_n)$-interpolant.
\end{remark}

The above discussion about the role of the Nokrane--Ransford result \cite[Theorem~1.1]{nokraneRansford:Slam01}
establishes how holomorphic correspondences are naturally related to the problem $(*)$. This is why we also consider
holomorphic correspondences in this paper. Indeed, 
the method that we employ to provide the proof of Theorem~\ref{T:3pt_nec}
motivated our investigation into finding a Schwarz lemma for holomorphic correspondences,
which are generalizations of algebroid multifunctions.
Before we present our result, we need a few definitions:

\begin{definition}\label{D:hol_corres}
Given domains $D_i\subseteq\C^n$, $i=1,2$, {\em a holomorphic correspondence} from $D_1$ to $D_2$
is an analytic subvariety $\Gamma$ of $D_1\times D_2$ of dimension $n$ such
that $\left.\pi_1\right|_\Gamma$ is surjective (where $\pi_1$ denotes the projection onto $D_1$).
\end{definition}

\noindent{A {\em proper holomorphic correspondence} $\Gamma$ from $D_1$ to $D_2$ is a holomorphic 
correspondence (as defined above) such that $\overline{\Gamma}\cap (D_1\times\bdy{D_2})=\emptyset$.
We refer the reader to Section~\ref{S:nott_comx_geom} for a discussion as to why holomorphic correspondences with
the latter property are called proper holomorphic correspondences.
A proper holomorphic correspondence $\Gamma$ from $D_1$ to $D_2$ also induces the following set-valued map:
\begin{equation}\label{E:img_corr}
\multF(z):=\{w\in D_2\,:\,(z,w)\in\Gamma\} \;\; \forall z\in D_1.
\end{equation}

The {\em Carath\'{e}odory pseudo-distance}, denoted by $C_{\OM}$, on a domain $\OM$ in $\C$ is defined by:
\begin{equation}\label{E:defn_carathdist}
C_{\Omega}(p,\,q):=\sup\{\mathcal{M}_{\D}(f(p),\,f(q))\,:\,f\in\hol(\Omega,\,\D)\}.
\end{equation}
The reader will notice that we have defined $C_{\OM}$ in terms of the M{\"o}bius distance rather than the hyperbolic
distance on $\D$. This is done purposely because {\em most} conclusions in metric geometry that rely on $C_{\OM}$ are essentially unchanged 
if $\mathcal{M}_{\D}$ is replaced by the hyperbolic distance on $\D$ in \eqref{E:defn_carathdist}, and because the  M{\"o}bius distance
arises naturally in the proof of our next theorem. We now present this theorem:

\begin{theorem}\label{T:Schwarzlemma_corres_corol}
Let $\OM$ be a bounded domain in $\C$ and let $\Gamma$ be a proper holomorphic correspondence
from $\D$ to $\OM$. Then for every $\zt_1,\zt_2\in\D$ we have:
\[
\hausd\big(\multF(\zt_1),\,\multF(\zt_2)\big)\leq
{\mathcal{M}_{\D}\left(\zt_1,\,\zt_2\right)}^{{1}/{n}},
\]
where $\hausd$ denotes the Hausdorff distance induced by $C_{\OM}$, and $n$ is the multiplicity
of $\Gamma$.
\end{theorem}

\noindent{Here, the multiplicity $n$ is as given by Lemma~\ref{L:corres_equa_zeroset}
(also see Remark~\ref{Rem:mult_corr}) below.}
The inequality above reduces to the distance-decreasing property for the Carath\'{e}odory pseudo-distance if $\Gamma$
is merely the graph of a holomorphic map $F:\D\lrarw\OM$. From this perspective, we can view Theorem~\ref{T:Schwarzlemma_corres_corol} 
as a Schwarz lemma for proper holomorphic correspondences. It turns out that algebroid multifunctions are
precisely the proper holomorphic correspondences from $\D$ to itself (see Lemma~\ref{L:corres_equa_zeroset}).
Hence, Theorem~\ref{T:Schwarzlemma_corres_corol} generalizes
\cite[Theorem~1.1]{nokraneRansford:Slam01} by Nokrane--Ransford.
\smallskip

The above theorem is the consequence of a more precise inequality, which we present in Section~\ref{S:SchLem_holcorres}
as Theorem~\ref{T:SchwLem_holcorres}. The proof of the latter theorem is closely related to the proof of Theorem~\ref{T:3pt_nec}.
The proof of Theorem~\ref{T:3pt_nec} is presented in Section~\ref{S:3pt_nec_proof}, while
the proof of Theorem~\ref{T:Schwarzlemma_corres_corol}  is presented in Section~\ref{S:SchLem_holcorres}.
\medskip

\section{Some remarks on the holomorphic functional calculus}\label{S:holo_fc}

A very essential part of our proofs below is the ability, given a domain $\Omega\subset\C$ and a matrix
$A\in M_n(\C)$, to define $f(A)$ in a meaningful way for each $f\in\hol(\Omega)$, provided
$\sigma(A)\subset\Omega$. Most readers will be aware that this is what is known as the holomorphic functional calculus.
We briefly recapitulate what the holomorphic functional calculus is so that we can make an observation
about the boundary regularity of $\Omega$\,---\,where $\Omega$ is as in the statement of
Theorems~{\ref{T:Schwarzlemma_corres_corol} and \ref{T:SchwLem_holcorres}}\,---\,which will be relevant to our proofs 
in Section~\ref{S:SchLem_holcorres}.
\smallskip

The discussion in this paragraph makes sense for any unital Banach algebra $\mathscr{A}$, where we 
denote the norm on $\mathscr{A}$ by $\|\bcdot\|$. Let $a\in\mathscr{A}$ and write 
\[
\mathsf{Hol}(a):=\text{the set of all functions holomorphic in some neighbourhood of $\sigma(a)$}.
\]
With the understanding that if $f,g\in\mathsf{Hol}(a)$, $f+g$ and $fg$ are defined and holomorphic 
on $\mathsf{dom}(f)\cap\mathsf{dom}(g)\supset\sigma(a)$, which endows $\mathsf{Hol}(a)$ with the structure
of a unital $\C$-algebra, the {\em holomorphic functional calculus} is an assignment
$\Theta_{a}:\,f\longmapsto f(a)$ with the following properties:
\begin{enumerate}
\item[(i)] $\Theta_{a}:\mathsf{Hol}(a)\lrarw\mathscr{A}$ is a $\C$-algebra homomorphism.
\item[(ii)] $\Theta_{a}({\text{id}}_{\C})=a$.
\item[(iii)] Let $\{f_{\nu}\}\subset\mathsf{Hol}(a)$ and suppose there is an open set $U\supset\sigma(a)$
such that $U\subseteq\mathsf{dom}(f_{\nu})$ for every $\nu\in\nat$. Suppose $f\in\mathsf{Hol}(a)$ is such that 
$f_{\nu}\to f$ uniformly on compact subsets of $U$. Then $\|f_{\nu}(a)-f(a)\|\to 0$ as $\nu\to\infty$.
\end{enumerate}
It is a basic result of the spectral theory of Banach algebras that an assignment $\Theta_{a}$ with the above 
properties exists.
\smallskip

We now specialize to the Banach algebra $M_{n}(\C)$. Fix $A\in M_{n}(\C)$. Then, it is well known
that (see \cite[Chapter~7, Section~1]{dunfordsch:Linopera88}) for any polynomial
$p(z)=\sum_{j=0}^m\alpha_jz^j$ such that 
\begin{equation}
p^{(j)}(\lambda)=f^{(j)}(\lambda) \;\; \forall j\,:\, 0\leq j\leq \nu(\lambda)-1 \ \text{and} \ \forall\lambda\in\sigma(A),
\label{E:defeq_funccal}
\end{equation}
where 
\[
\nu(\lambda):=\min\Big\{k\in\nat\,:\,
\text{Ker}(\lambda\mathbb{I}-A)^{k+1}=\text{Ker}(\lambda\mathbb{I}-A)^{k}\Big\}, \ \lambda\in\sigma(A),
\]
the assignment 
\[
\Theta_{A}(f)=f(A):=\sum_{j=0}^m\alpha_j A^j
\]
has the properties (i), (ii) and (iii) above. 
Note that, for $\lambda\in\sigma(A)$, $\nu(\lambda)$ is the exponent of $(z-\lambda)$ in the minimal polynomial
of $A$. Now, given a non-empty open set $\OM\subset\C$ and $A\in M_n(\C)$ such that $\sigma(A)\subset\OM$,
one defines
\[
f(A):=\Theta_{A}(f)\;\;\forall f\in\hol(\OM).
\]
By the foregoing discussion, we need to make {\em no assumptions} about the boundary of  $\OM$
in defining $f(A)$, $f\in\hol(\OM)$, such that
the assignment $\hol(\OM)\ni f\mapsto f(A)$ ({\em provided} $\sigma(A)\subset\OM$) behaves ``naturally''.
We consider this point relevant to make because in treatments of the assignment $\hol(\OM)\ni f\mapsto f(a)$
in certain books, $a$ belonging to a {\em general} unital Banach algebra $\mathscr{A}$, this assignment
is defined via a Cauchy integral and with certain conditions imposed
on $\bdy{\OM}$ when $\OM\varsubsetneq\C$. A rephrasing of the above point in a manner that is more precise and 
relevant to the proofs in Section~\ref{S:SchLem_holcorres} is as follows.

\begin{remark}\label{Rem:matpara_func}
Let $\Omega$ be a non-empty open set in $\C$ and let $S_{n}(\Omega):=\{A\in M_{n}(\C)\,:\,\sigma(A)
\subset\Omega\}$, $n\geq 2$. Then for each $f\in\hol(\Omega)$ and $A\in S_{n}(\Omega)$, we can define
$f(A)$ such that $f(A)$\,---\,fixing $A\in S_{n}(\Omega)$ and writing $f(A):=\Theta_{A}(f)$\,---\,has
the properties (i)--(iii) above (taking $\mathscr{A}=M_{n}(\C)$, $a=A$
and with $\hol(\Omega)$ and $\Omega$ replacing $\mathsf{Hol}(a)$ and $U$, respectively)
$\forall f\in \hol(\Omega)$ \textbf{without} any conditions
on $\bdy{\Omega}$ or on whether $f\in\hol(\Omega)$ extends to $\partial{\Omega}$.
With $\OM$ as above and $A\in S_{n}(\Omega)$, the assignment $\hol(\OM)\ni f\mapsto f(A)$ will \textbf{also}
be called the holomorphic functional calculus in our discussions below.
\end{remark}
We end this section by stating the Spectral Mapping Theorem. When we invoke it in subsequent sections,
it will be for the Banach algebra $\mathscr{A}=M_{n}(\C)$.
\begin{result}[Spectral Mapping Theorem]
Let $\mathscr{A}$ be a unital Banach algebra. Then for every $f\in\mathsf{Hol}(a)$ and $a\in\mathscr{A}$ we have
\[
\sigma(f(a))=f(\sigma(a)).
\]
\end{result} 
\smallskip

\section{Minimal polynomials under the holomorphic functional calculus}\label{S:minpo_holo_fc}
In this section, we develop the key matricial tool needed in establishing Theorem~\ref{T:3pt_nec}, which is 
the computation of the minimal polynomial for $f(A)$, given $f\in\hol(\D)$ and $A\in\Omega_n$, $n\geq 2$. 
This is the content of Theorem~\ref{T:minpo_holo_func_anal}. We begin with a few lemmas which will help us to
establish  Theorem~\ref{T:minpo_holo_func_anal}. In what follows, given integers $p < q$, $\intgR{p}{q}$
will denote the set of
integers $\{p, p+1,\dots, q\}$. Recall:
$[\bcdot]$ denotes the greatest-integer function. Also, given $A\in M_{n}(\C)$, we will
denote its minimal polynomial by $\minpo{A}$.
\smallskip

Let $n\geq 2$. Given $(\alpha_1,\alpha_2,\dots,\alpha_{n-1})\in\C^{n-1}$, we define
$l(\alpha_1,\alpha_2,\dots,\alpha_{n-1})\in\intgR{1}{n}$ by
\[
l(\alpha_1,\alpha_2,\dots,\alpha_{n-1})\,:=\,
\begin{cases}
n, &\text{if $\alpha_j = 0 \ \forall j\in\intgR{1}{n-1}$},\\
\min\{j\in \intgR{1}{n-1}\,:\,\alpha_j\neq 0\}, &\text{otherwise}.
\end{cases}
\]

\begin{lemma}\label{L:minmo_lincomb_nilpo}
Let $(\alpha_0,\alpha_1,\dots,\alpha_{n-1})\in\C^{n}$, $n\geq 2$. Let 
$A\,=\,\sum_{j=0}^{n-1}\alpha_jN^{j}$, where $N\in M_{n}(\C)$ is the
nilpotent matrix of degree $n$ given by $(\KDel_{i+1,\,j})_{i,\,j\,=\,1}^n$, $\KDel_{\mu,\,\nu}$ being
the Kronecker symbol. Then the minimal polynomial for $A$ is given by:
\begin{equation}\label{E:minpoeq_lincomb_nilpo}
\minpo{A}(t)\,=\,(t-\alpha_{0})^{[(n-1)/{l(\alpha_1,\alpha_2,\dots,\alpha_{n-1})}]+1},
\end{equation}
where $l(\alpha_1,\alpha_2,\dots,\alpha_{n-1})$ is as defined above.
\end{lemma}

\begin{proof}
The proof consists of two cases.

\noindent{\bf Case 1.} $\alpha_j\,=\,0\,\,\forall j\in\intgR{1}{n-1}$.

\noindent This implies that $l(\alpha_1,\alpha_2,\dots,\alpha_{n-1})\,=\,n$, and
hence $[(n-1)/{l(\alpha_1,\alpha_2,\dots,\alpha_{n-1})}]\,=\,0$.
The minimal polynomial in this case clearly is $(t-\alpha_0)$.
This establishes \eqref{E:minpoeq_lincomb_nilpo} in this case.
\smallskip

\noindent{\bf Case 2.} {\em $\alpha_j\,\neq\,0$ for some $j\in\intgR{1}{n-1}$.}

\noindent We write $l\,\equiv\,l(\alpha_1,\alpha_2,\dots,\alpha_{n-1})$.
Then $A-\alpha_{0}I\,=\,\alpha_lN^l+\cdots+\alpha_{n-1}N^{n-1}$.
Hence
\[
(A-\alpha_{0}I)^{m}={\alpha_l}^{m}N^{lm}+(\text{terms in $N^k$ with $k>lm$}).
\]
We observe thus that the power of $(t-\alpha_0)$ in $\minpo{A}(t)$ must be the least integer
$m$ such that $ml\geq n$. It is elementary to see that that integer is
$[(n-1)/l]+1$. 
\end{proof}

Let $a\in \C$ and let $g$ be a function that is holomorphic in a neighbourhood of $a$. Then, 
$\mathsf{ord}_{a}g$ will denote the order of vanishing of $g$ at $a$.
Recall: this means that $\mathsf{ord}_{a}g$ is the least non-negative
integer $j$ such that $g^{(j)}(a)\,\not=\,0$ (hence, $\mathsf{ord}_{a}g = 0$ if $g$
does not vanish at $a$).

\begin{lemma}\label{L:minpo_Jordan}
Let $\lambda\in\D$ and let $f\in\hol(\D)$ be a non-constant function. Let $J_{n}(\lambda)$ represent
the $n\times n$ Jordan matrix associated to $\lambda$, $n\geq 2$, and $f(J_{n}(\lambda))$ be the
matrix given by the holomorphic functional calculus. Then the minimal polynomial of $f(J_{n}(\lambda))$ is given by
\[
\minpo{f(J_{n}(\lambda))}(t)\,=\,\left(t-f(\lambda)\right)
^{\intf[]{n-1}{\mathsf{ord}_{\lambda}{f'}+1}+1}.
\]
\end{lemma}
\begin{proof}
We begin by noting that, while \eqref{E:defeq_funccal} gives an expression for $f(J_{n}(\lambda))$
in terms of the exponent of $(t-\lambda)$ in $\minpo{J_{n}(\lambda)}$, our task is to determine the analogous 
exponent in $\minpo{f(J_n(\lambda))}$, for which  \eqref{E:defeq_funccal} is not immediately helpful.
\smallskip

Let $R$ be such that $|\lambda|<R<1$. Then, the power series expansion of 
$f$ 
\[
f(z)\,:=\,\sum_{k\in\nat}\intf{}{}{f^{(k)}(0)}{k!}{z^k}\,\,
\text{converges absolutely at each}\,\,z\in D(0;\,R).
\]
Then by elementary properties (see Section~\ref{S:holo_fc}) of the holomorphic functional calculus we get
\begin{equation}\label{E:powerseries_Jordan}
f(J_{n}(\lambda))\,=\,\sum_{k\in\nat}\intf{}{}{f^{(k)}(0)}{k!}(J_{n}(\lambda))^{k}.
\end{equation}
Note that $J_{n}(\lambda)=\lambda\mathbb{I}+N$, where $N$ is the nilpotent matrix as in
Lemma~\ref{L:minmo_lincomb_nilpo}. We can use the binomial expansion to get
\[
(J_{n}(\lambda))^{k}\,=\,(\lambda\mathbb{I}+N)^k\,=\,\sum_{j=0}^{p(k)}
\facto{k}{j}{\lambda}^{k-j}N^j,
\]
where $p(k):=\min(k,\,n-1)$ and $k\in\nat$. Hence \eqref{E:powerseries_Jordan} becomes
\begin{equation}\label{E:doublesum_nilpo}
f(J_{n}(\lambda))\,=\,\sum_{k\in\nat}\intf{}{}{f^{(k)}(0)}{k!}\sum_{j=0}^{p(k)}
\facto{k}{j}{\lambda}^{k-j}N^j.
\end{equation}
The coefficient of $N^j$, $0\leq j\leq n-1$, in \eqref{E:doublesum_nilpo} is 
\[
\sum_{k\geq j}\intf{}{}{f^{(k)}(0)}{k!}\facto{k}{j}{\lambda}^{k-j}\,=\,
\sum_{k\geq j}\intf{}{}{f^{(k)}(0)}{(k-j)!\,{j!}}{\lambda}^{k-j}\,=\,\intf{}{}{f^{(j)}(\lambda)}{j!},\,\,j\in\nat.
\]
Using the fact that $N^n=0$, we get 
\[
f(J_{n}(\lambda))\,=\,\sum_{j=0}^{n-1}\intf{}{}{f^{(j)}(\lambda)}{j!}N^{j}.
\]
From Lemma~\ref{L:minmo_lincomb_nilpo} we have $\minpo{f(J_{n}(\lambda))}(t)\,=\,(t-f(\lambda))^{m}$, 
where 
\[
m\,=\,\intf[]{n-1}{l(f'(\lambda),f''(\lambda),\dots,f^{(n-1)}(\lambda))}+1.
\]
If $\mathsf{ord}_{\lambda}{f'}\leq n-2$,
then $l(f'(\lambda),f''(\lambda),\dots,f^{(n-1)}(\lambda))\,=\,\mathsf{ord}_{\lambda}{f'}+1$,
else $\mathsf{ord}_{\lambda}{f'}+1$ and $l(f'(\lambda),f''(\lambda),\dots,f^{(n-1)}(\lambda))>(n-1)$.
In both the cases we have:
\[
\intf[]{n-1}{l(f'(\lambda),f''(\lambda),\dots,f^{(n-1)}(\lambda))}\,=\,
\intf[]{n-1}{\mathsf{ord}_{\lambda}{f'}+1}.
\]
From the last two expressions, the lemma follows.
\end{proof}

\begin{lemma}\label{L:minmo_Jordanblocks}
Let $\lambda\in\D$ and $f\in\hol(\D)$, $f$ non-constant. Let $n_1\leq n_2\leq\cdots\leq n_{q}$
be a sequence of positive integers. Let $J\,=\,\oplus_{i=1}^{q}J_{n_i}(\lambda)$, where $J_{n_i}(\lambda)$
represents the $n_i\times n_i$ Jordan block associated to $\lambda$. Then the minimal polynomial 
for $f(J)$ is given by:
\[
\minpo{f(J)}(t)\,=\,\left(t-f(\lambda)\right)
^{\intf[]{n_{q}-1}{\mathsf{ord}_{\lambda}{f'}+1}+1}.
\]
\end{lemma}
\begin{proof}
Note that $f(J)\,=\,\oplus_{i=1}^{q}f(J_{n_i}(\lambda))$. If $n_i = 1$, for $i=1,\ldots,q$, then the following is
obvious; else Lemma~\ref{L:minpo_Jordan} gives us
\begin{equation}\label{E:minpo_summand}
\minpo{f(J_{n_i}(\lambda))}(t)\,=\,\left(t-f(\lambda)\right)
^{\intf[]{n_i-1}{\mathsf{ord}_{\lambda}{f'}+1}+1}\,\,\forall i\in\intgR{1}{q}.
\end{equation}
For each $i$, we also have
\begin{equation}\label{E:ineq_expo}
\intf[]{n_i-1}{\mathsf{ord}_{\lambda}{f'}+1}+1\leq\,
\intf[]{n_q-1}{\mathsf{ord}_{\lambda}{f'}+1}+1.
\end{equation}
Now the minimal polynomial for a matrix that is a finite direct sum of matrices is the least common 
multiple (in the ring $\C[t]$) of the minimal polynomials of the direct summands. This, together
with \eqref{E:minpo_summand}
and \eqref{E:ineq_expo}, establishes the lemma.
\end{proof}

\begin{theorem}\label{T:minpo_holo_func_anal}
Let $A\in\Omega_n$, $n\geq 2$, and let $f\in\hol(\D)$ be a non-constant function. Suppose that the minimal
polynomial for $A$ is given by
\[
\minpo{A}(t)\,=\,\prod_{\lambda\in\sigma(A)}(t-\lambda)^{m(\lambda)}.
\]
Then the minimal polynomial for $f(A)$ is given by
\[
\minpo{f(A)}(t)\,=\,\prod_{\nu\in f(\sigma(A))}(t-\nu)^{k(\nu)},
\]
where, $k(\nu)=\max\left\{\intf{[}{]}{m(\lambda)-1}
{\mathsf{ord}_{\lambda}{f'}+1}+1:\,\,\lambda \in \sigma(A)\cap f^{-1}\{\nu\}\right\}.
$
\end{theorem}

\begin{proof}
By the Spectral Mapping Theorem the minimal polynomial for $f(A)$ is given by
$\prod_{\nu\in f(\sigma(A))}(t-\nu)^{k(\nu)}$ for some $k(\nu)\in\nat$. We must now show that
$k(\nu)$ are as stated above. Let $\mathfrak{S}(\nu)$, for each $\nu\in f(\sigma(A))$,
denote the set
\[
\mathfrak{S}(\nu)\,:=\,\{\lambda\in\sigma(A)\,:\,f(\lambda)\,=\,\nu\}.
\]
Then $\{\mathfrak{S}(\nu)\,:\,\nu\in f(\sigma(A))\}$ gives a partition of $\sigma(A)$.
\smallskip

The Jordan canonical form tells us that $\exists S\in GL_{n}(\C)$ such that 
\begin{equation}\label{E:Jordanform_A}
A\,=\,S\left[\oplus_{\nu\in f(\sigma(A))}\left[\oplus_{\lambda\in\mathfrak{S}(\nu)}
\left[\oplus_{i=1}^{q(\lambda)}J_{n_{i}^\lambda}(\lambda)\right]\right]\right]S^{-1},
\end{equation}
where $\left\{J_{n_{i}^\lambda}(\lambda)\,:\,i\in\intgR{1}{q(\lambda)}\right\}$ is the Jordan block-system
associated to $\lambda$ such that $n_{1}^\lambda\leq n_{2}^\lambda\leq\cdots\leq n_{q(\lambda)}^\lambda$.
Now from \eqref{E:Jordanform_A} and from the basic properties of the holomorphic functional calculus
we get
\begin{equation}\label{E:minpo_Jordanform_A}
\minpo{f(A)}\,=\,{\boldsymbol{\mathsf{M}}}\left(\oplus_{\nu\in f(\sigma(A))}\left[\oplus_{\lambda\in\mathfrak{S}(\nu)}
\left[f\left(\oplus_{i=1}^{q(\lambda)}J_{n_{i}^\lambda}(\lambda)\right)\right]\right]\right)
\end{equation}
(we will sometimes write $\minpo{B}$ as ${\boldsymbol{\mathsf{M}}}(B)$ for convenience).
Notice that if $\nu_1\not=\nu_2\in f(\sigma(A))$, the matrices
$\oplus_{\lambda\in\mathfrak{S}(\nu_j)}
f\left(\oplus_{i=1}^{q(\lambda)}J_{n_{i}^\lambda}(\lambda)\right),\ j=1,2$, have
$\{\nu_{1}\}$ and $\{\nu_{2}\}$, respectively, as their spectra.
Hence the minimal polynomials of these are relatively
prime to each other. This implies that (from \eqref{E:minpo_Jordanform_A} above):
\begin{equation}\label{E:minpo_sum_Jordanblocks}
\minpo{f(A)}\,=\,\prod_{\nu\in f(\sigma(A))}{\boldsymbol{\mathsf{M}}}
\left(\left[\oplus_{\lambda\in\mathfrak{S}(\nu)}
\left[f\left(\oplus_{i=1}^{q(\lambda)}J_{n_{i}^\lambda}(\lambda)\right)\right]\right]\right).
\end{equation}
The above is the consequence of the fact that the minimal polynomial of a direct sum of matrices
is the least common multiple of the minimal polynomials of the individual matrices. This also implies 
that
\begin{equation}\label{E:minpo_lcm}
{\boldsymbol{\mathsf{M}}}\left(\left[\oplus_{\lambda\in\mathfrak{S}(\nu)}
\left[f\left(\oplus_{i=1}^{q(\lambda)}J_{n_{i}^\lambda}(\lambda)\right)\right]\right]\right)\,=\,
{\sf lcm}\left\{{\boldsymbol{\mathsf{M}}}
\left(f\left(\oplus_{i=1}^{q(\lambda)}J_{n_{i}^\lambda}(\lambda)\right)\right)\,:\,
\lambda\in\mathfrak{S}(\nu)\right\}.
\end{equation}
For a fixed $\lambda\in\mathfrak{S}(\nu)$, recall that
$n_{1}^\lambda\leq n_{2}^\lambda\leq\cdots\leq n_{q(\lambda)}^\lambda$. Furthermore, 
$n_{q(\lambda)}^\lambda\,=\,m(\lambda)$, $m(\lambda)$ being the multiplicity of $\lambda$ in
$\minpo{A}$. Putting all of this together with Lemma~\ref{L:minmo_Jordanblocks} we have:
\begin{equation}\label{E:minpo_holo_jordan}
{\boldsymbol{\mathsf{M}}}\left(f\left(\oplus_{i=1}^{q(\lambda)}J_{n_{i}^\lambda}(\lambda)\right)\right)
(t)\,=\,\left(t-\nu\right)^{\intf{[}{]}{m(\lambda)-1}{\mathsf{ord}_{\lambda}{f'}+1}+1}\,\,
\forall\lambda\in\mathfrak{S}(\nu).
\end{equation}
Now, \eqref{E:minpo_holo_jordan} and \eqref{E:minpo_lcm} together give us:
\[
{\boldsymbol{\mathsf{M}}}\left(\left[\oplus_{\lambda\in\mathfrak{S}(\nu)}
\left[f\left(\oplus_{i=1}^{q(\lambda)}J_{n_{i}^\lambda}(\lambda)\right)\right]\right]\right)\,=\,
(t-\nu)^{k(\nu)},
\]
where $k(\nu)$ is as stated in our theorem. The above, in view of \eqref{E:minpo_sum_Jordanblocks},
gives the result.
\end{proof}
\smallskip

\section{Two fundamental lemmas}\label{S:princ_lemma}

In this section, we state two fundamental and closely related lemmas.
Lemma~\ref{L:fl2} serves as the link between the two main results of this paper. 
Both lemmas are simple, once we appeal to Vesentini's theorem. We begin by stating this result.

\begin{result}[Vesentini, \cite{vst:shSprd68}]\label{Res:satz_vst}
Let $\mathscr{A}$ be a complex, unital Banach algebra and let $\rho(x)$ denote the spectral radius
of any element $x\in\mathscr{A}$. Let $f\in\hol(\D,\,\mathscr{A})$. Then, the function
$\zt\longmapsto\rho(f(\zt))$ is subharmonic on $\D$.
\end{result}

\begin{lemma}\label{L:fl1}
Let $\Phi\in\hol(\D,\,\overline{\Omega}_n)$ be such that there exists a
$\theta_0\in\mathbb{R}$ and $\zt_0\in\D$ satisfying $e^{i\theta_0}\in\sigma(\Phi(\zt_0))$.
Then $e^{i\theta_0}\in\sigma(\Phi(\zt))$ for all $\zt\in\D$.
\end{lemma}

\begin{proof}
Define $\widetilde{\Phi}\in\hol(\D,\,M_n(\C))$ by
\[
\widetilde{\Phi}(\zt)\,=\,\Phi(\zt)+e^{i\theta_0}\mathbb{I},\,\,\forall \zt\in\D.
\]
By Result~\ref{Res:satz_vst}, $\rho\circ\widetilde{\Phi}$ is subharmonic on $\D$. Notice 
that for each $\zt\in\D$ we have $\rho\circ\widetilde{\Phi}(\zt)\leq 2$,
and $\rho(\widetilde{\Phi}(\zt_0))=2$. By the maximum
principle for subharmonic functions it follows that $\rho\circ\widetilde{\Phi}\equiv 2$. As 
$\sigma(\widetilde{\Phi}(\zt))=e^{i\theta_0}+\sigma(\Phi(\zt))$
and $\sigma(\Phi(\zt))\subseteq\overline{\D}$, this implies that
$e^{i\theta_0}\in\sigma(\Phi(\zt))\,\, \forall \zt\in\D$. Hence the lemma.
\end{proof}

The next lemma is, essentially, a fragment of a proof in \cite[Section 3]{gb:itplSpUb07}. However,
since it requires a non-trivial fact\,---\,i.e., the plurisubharmonicity 
of the spectral radius\,---\,we provide a proof.

\begin{lemma}\label{L:fl2}
Let $F\in\hol(\D,\,\Omega_n)$ be such that $F(0)=0$. Then, there exists
$G\in\hol(\D,\,\overline{\Omega}_n)$ such that $F(\zt)\,=\,\zt\,G(\zt)$ for all $\zt\in\D$.
\end{lemma}

\begin{proof}
As $F(0)=0$, there exists $G\in\hol(\D,\,M_n(\C))$ such that $F(\zt)\,=\,\zt\,G(\zt)\,\,\forall\zt\in\D$. 
Fix a $\zt\in\D\setminus\{0\}$ and let $R\in (0,\,1)$ be such that $R>|\zt|$. Then on the circle
$|w|\,=\,R$ we have 
\begin{align}
\rho(F(w))\,&=\,R\,\rho(G(w))\nonumber\\
\implies \rho(G(w))\,&=\,\intf{}{}{\rho(F(w))}{R}<\intf{}{}{1}{R}\,\,\forall w:\,|w|=R,\label{E:maxSp}
\end{align}
where $\rho$ denotes the spectral radius. We again appeal to Vesentini's Theorem. Subharmonicity 
of $\rho\circ G$, the maximum principle and \eqref{E:maxSp} give us:
\[
\rho(G(\zt))<\intf{}{}{1}{R}\,\,(\text{recall that $|\zt|<R$}).
\]
By taking $R\nearrow 1$, and from the fact that $\zt\in\D$ was arbitrary,
we get $\rho(G(\zt))\leq 1\,\,\forall\zt\in\D$.
This is equivalent to $G\in\hol(\D,\,\overline{\Omega}_n)$.
\end{proof}
\medskip

\section{Some notations and results in basic complex geometry}\label{S:nott_comx_geom}

This section is devoted to a couple of results in the geometry and function theory in the
holomorphic setting that are relevant to our proof of Theorem~\ref{T:Schwarzlemma_corres_corol}.
Our first result pertains to the structure of a holomorphic correspondence $\Gamma$ from $\D$ to
$\Omega$ with the properties as in Theorem~\ref{T:Schwarzlemma_corres_corol}. For this, we need the
following standard result (see \cite[Section~3.1]{chirka:comx_ana_set89}, for instance).

\begin{result}\label{Res:crit_propproj}
Let $\OM_1\subsetneq X$ and $\OM_2\subsetneq Y$ be open subsets in topological
spaces $X$ and $Y$ respectively, with $\overline{\OM}_2$ being compact. Let $A$ be a closed subset in
$\OM_1\times \OM_2$. Then the restriction to $A$ of the projection $(x,\,y)\longmapsto x$ is proper if and only if
$\overline{A}\cap(\OM_1\times\partial{\OM_2})=\emptyset$.
\end{result}

Owing to the above result, any holomorphic correspondence $\Gamma$ from $D_1$ to
 $D_2$, which are domains in $\C^{n}$, such that
$\overline{\Gamma}\cap(D_1\times\partial{D_2})=\emptyset$ is called a 
\emph{proper holomorphic correspondence}. We can now state and prove the result that we need.
This result is deducible, in essence, from \cite[Section~4.2]{chirka:comx_ana_set89}.
However, since that discussion pertains to a much more general setting,
we provide a proof in the set-up that we are interested in.

\begin{lemma}\label{L:corres_equa_zeroset}
Let $\Omega$ be a bounded domain in $\C$ and let $\Gamma$ be a proper holomorphic correspondence from
$\D$ to $\Omega$. Then there exist an $n\in\Z_{+}$ and functions $a_1,\dots,a_n\in\hol(\D)$
such that
\begin{equation}
\Gamma=\Big\{(z,w)\in\D\times\Omega\,:\,w^n+\sum_{j=1}^n(-1)^ja_j(z)w^{n-j}=0\Big\}.
\end{equation} 
\end{lemma} 

\begin{proof}
Since $\left.\pi_1\right|_{\Gamma}$ is proper, it follows from the elementary theory of proper
holomorphic maps that there exists a discrete set $\mathcal{A}\subset\D$ and an $n\in\Z_{+}$ such
that $(\Gamma\setminus{\pi_1^{-1}(\mathcal{A})},\,\D\setminus\mathcal{A},\,\left.\pi_1\right|_{\Gamma})$
is an $n$-fold analytic covering. Thus, given any $p\in\D\setminus\mathcal{A}$, there exist an open
neighborhood $V_p$ such that $p\in V_p\subset\D\setminus\mathcal{A}$, and $n$
holomorphic inverse branches of $\left.\pi_1\right|_{\Gamma}$ ;
$(\pi_1)^{-1}_{1,\,p},\dots,(\pi_1)^{-1}_{n,\,p}\in\hol(V_p,\,\Gamma)$; such that the images of
$(\pi_1)^{-1}_{j,\,p}$, $j=1,\ldots,n$, are disjoint.
\smallskip

Let $\mathscr{S}_{j}$ denote the $j$-th elementary symmetric polynomial in $n$ symbols. Define:
\[
a_j(z):=\mathscr{S}_{j}((\pi_1)^{-1}_{1,\,z}(z),\dots,(\pi_1)^{-1}_{n,\,z}(z))
 \ \forall z\in\D\setminus\mathcal{A}, \ 1\leq j\leq n.
\]
Clearly, $a_j$ does not depend on the order in which $\{(\pi_1)^{-1}_{k,\,p}\}_{k\in\intgR{1}{n}}$
are labeled, whence it is well-defined. Now, fix a $p\in\D\setminus\mathcal{A}$. Provisionally,
for $z\in V_p$, let us define:
\[
(\pi_1)^{-1}_{j,\,z}:=\text{the holomorphic inverse branch of $\left.\pi_1\right|_{\Gamma}$
that maps $z$ to $(\pi_1)^{-1}_{j,\,p}(z),$}
\]
which is well-defined, because the images of $(\pi_1)^{-1}_{j,\,p}$, $j=1,\dots,n$, are disjoint.
By the same reason $\{(\pi_1)^{-1}_{k,\,z}(z):1\leq k\leq n\}=\{(\pi_1)^{-1}_{j,\,p}(z):1\leq j\leq n\}$.
From the last two assertions, we have 
\[
a_j(z)=\mathscr{S}_{j}((\pi_1)^{-1}_{1,\,p}(z),\dots,(\pi_1)^{-1}_{n,\,p}(z)) \ \forall z\in V_p.
\]
This tells us that $a_j$ is $\C$-differentiable at each $p\in\D\setminus\mathcal{A}$, whence 
$a_j\in\hol(\D\setminus\mathcal{A})$. It is easy to see that, for each $q\in\mathcal{A}$
\begin{equation}\label{E:exiL_excep}
\lim_{\D\setminus\mathcal{A}\,\ni z\to q}a_j(z)=\sum_{(i_1,\dots,i_j)\in\mathscr{I}_{n}(j)}
\prod\nolimits_{w_{i_k}\in\bv{(\pi_1^{-1}\{q\}\cap\Gamma)},\,{1\leq k\leq j}} \ w_{i_k},
\end{equation}
where $\mathscr{I}_n(j)$ denotes the set of all increasing $j$-tuples of $\intgR{1}{n}$,
$j=1,\ldots,n$, and
\begin{align*}
\bv{(\pi_1^{-1}\{q\}\cap\Gamma)}:={}& \text{an enumeration of the {\bf list} of elements
in $\pi_1^{-1}\{q\}\cap\Gamma$} \\
{}& \text{repeated according to intersection multiplicity.}
\end{align*}
From \eqref{E:exiL_excep} and Riemann's removable singularities theorem, we conclude that 
$a_j$ extends to a holomorphic function, $j\in\intgR{1}{n}$.
Furthermore, by the definition of $\left.a_j\right|_{\D\setminus\mathcal{A}}$ and by \eqref{E:exiL_excep},
we conclude, from Vieta's formulas that, fixing a $z_0\in\D$:
\begin{align*}
{}& \text{the {\bf list}, repeated according to multiplicity, of the zeros of
$w^n+\sum_{j=1}^n(-1)^ja_j(z_0)w^{n-j}$}\\
{}&=\bv{\multF(z_0)}.
\end{align*}
As $z_0\in\D$ was arbitrary, the result follows. 
\end{proof}

\begin{remark}\label{Rem:not_mltply}
We have used a notation in our proof of Lemma~\ref{L:corres_equa_zeroset}\,---\,see \eqref{E:exiL_excep} and
the clarifications that follow\,---\,that will be useful in later discussions/calculations. Namely: if $S$ is a non-empty set,
and there is a multiplicity associated to each $s\in S$, then we shall use the notation $\bv{S}$ to denote the {\bf list}
of elements of $S$ repeated according to their multiplicity.
\end{remark}

\begin{remark}\label{Rem:mult_corr}
The positive integer $n$ that appears in
the above lemma is known as \emph{the multiplicity of $\Gamma$}.
In general, when we have a proper holomorphic correspondence $\Gamma$ from $\Omega_1$ to $\Omega_2$ with
$\dim(\Gamma)=\dim(\Omega_1)$, then there exists an analytic variety
$\mathcal{A}\subset\Omega_1$ with $\dim{\mathcal{A}}<\dim(\Gamma)$
such that the cardinality of  $\pi_1^{-1}\{z\}\cap\Gamma=k$ for all $z\in\Omega\setminus\mathcal{A}$
(see \cite[Section~3.7]{chirka:comx_ana_set89}).
This generalizes the notion of multiplicity to higher dimensions.
\end{remark}

We shall now look at an extremal problem associated to the Carath\'{e}odory
pseudo-distance $C_{\Omega}$ on the domain $\Omega$ in $\C$. Recall the discussion in Section~\ref{S:intro}
about the Carath\'{e}odory pseudo-distance, and the reasons that we prefer using the following definition:
\begin{align}
C_{\Omega}(p,\,q)&:=\sup\{\mathcal{M}_{\D}(f(p),\,f(q))\,:\,f\in\hol(\Omega,\,\D)\}\nonumber\\
&=\sup\{|\,f(q)\,|\,:\,f\in\hol(\Omega,\,\D)\,:\,f(p)=0\}.\label{E:alt_def_cara}
\end{align}
The equality in \eqref{E:alt_def_cara} is due to the fact that the automorphism group of $\D$
acts transitively on $\D$ and the M{\"o}bius distance is invariant under its action.
Applying Montel's Theorem, it is easy to see that there exists a function 
$g\in\hol(\Omega,\,\D)$ such that $g(p)=0$ and $g(q)=C_{\Omega}(p,\,q)$. Such a function is
called an \emph{extremal solution} for the extremal problem determined by \eqref{E:alt_def_cara}.\vspace{1mm}
\smallskip

We will always consider domains $\Omega$ in $\C$ for which $H^{\infty}(\Omega)$, the set of all 
bounded holomorphic functions in $\Omega$, separates points in $\Omega$. For such domains the Carath\'{e}odory
pseudo-distance clearly is a distance. Moreover it turns out that for such domains there is a unique extremal
solution (see the last two paragraphs in \cite{fsh:shwaLemInnfun69}).
In Section~\ref{S:SchLem_holcorres} (since the domains considered there are bounded), we will always denote by $G_{\Omega}(p,\,q;\,\bcdot{•})$
the unique extremal solution determined by the extremal problem \eqref{E:alt_def_cara}.
\medskip

\section{The proof of Theorem~\ref{T:3pt_nec}}\label{S:3pt_nec_proof}

This section is largely devoted to the proof of Theorem~\ref{T:3pt_nec}.
Closely related to it is our example\,---\,referred to in Section~\ref{S:intro}\,---\,that compares
the necessary condition given by Theorem~\ref{T:3pt_nec} with other necessary conditions for
the existence of a $3$-point interpolant from $\D$ to $\OM_n$. This example is presented after
our proof.

\subsection{The proof of Theorem~\ref{T:3pt_nec}}\label{SS:3pt_nec_proof}

Let $F\in\hol(\D,\,\Omega_n)$ be such that $F(\zt_j)=W_j$, $j\in\{1,2,3\}$. Fix $k\in\{1,2,3\}$ and
write:
\[
\widetilde{F_k}:=B_k\circ F\circ\psi_k^{-1},
\]
where $B_k$, $\psi_k$ are as described before the statement of Theorem~\ref{T:3pt_nec}.
Notice that $\widetilde{F_k}\in\hol(\D,\,\Omega_n)$ and satisfies
$\widetilde{F_k}(\psi_k(\zt_{L(k)}))=B_k(W_{L(k)}),\,\widetilde{F_k}(\psi_k(\zt_{G(k)}))
=B_k(W_{G(k)}),\,\widetilde{F_k}(0)=0$. By Lemma~\ref{L:fl2}, we get
\begin{equation}\label{E:fact_aux_intp}
\widetilde{F_k}(\zt)=\zt\,\widetilde{G_k}(\zt)\,\,\forall\zt\in\D,\,\,\text{ for some 
$\widetilde{G_k}\in\hol(\D,\,\overline{\Omega}_n)$}.
\end{equation}
Two cases arise:
\smallskip

\noindent{\bf Case 1.} $\widetilde{G_k}(\D)\subset\Omega_n$.

\noindent In view of \eqref{E:fact_aux_intp}, we have 
\begin{equation}\label{E:2pt_redut_intp}
\widetilde{G_k}\left(\psi_k(\zt_{L(k)})\right)\,=\,W_{L(k),\,k},\,\,
\widetilde{G_k}\left(\psi_k(\zt_{G(k)})\right)\,=\,W_{G(k),\,k},
\end{equation}
where $W_{L(k),\,k}:={B_k(W_{L(k)})}\big{/}{\psi_{k}(\zt_{L(k)})}$ and
$W_{G(k),\,k}:={B_k(W_{G(k)})}\big{/}{\psi_{k}(\zt_{G(k)})}$.
Now by Result~\ref{Res:gbintS}, a necessary condition for \eqref{E:2pt_redut_intp} is
\begin{equation}\label{E:ineq_2pt_redut}
\max\left\{\sub{\eta\in\sigma\left(W_{L(k),\,k}\right)}{\max}|\,b_{G(k),\,k}(\eta)\,|,\,
\sub{\eta\in\sigma\left(W_{G(k),\,k}\right)}{\max}|\,b_{L(k),\,k}(\eta)\,|\right\}
\leq\mathcal{M}_{\D}\left(\zt_{G(k)},\,\zt_{L(k)}\right),
\end{equation}
where $b_{L(k),\,k}$ and $b_{G(k),\,k}$ denote the minimal Blaschke product corresponding to the 
matrices $W_{L(k),\,k}\,,\,\,W_{G(k),\,k}$. Given the definitions of the latter matrices, we will need 
Theorem~\ref{T:minpo_holo_func_anal} to determine  $b_{L(k),\,k}$, $b_{G(k),\,k}$. By this theorem, we have
\begin{align}
b_{L(k),\,k}(t)&=\prod_{\nu\in\sigma\left(B_k(W_{L(k)})\right)}
\!\!\!{\intf{(}{)}{t-{\nu}/{\psi_k(\zt_{L(k)})}}{1-\overline{{\nu}/{\psi_k(\zt_{L(k)})}}t}}^{q(\nu,\,L(k),\,k)}\,\,
\label{E:minpo_2pt_redut1}\\ 
b_{G(k),\,k}(t)&=\prod_{\nu\in\sigma\left(B_k(W_{G(k)})\right)}
\!\!\!{\intf{(}{)}{t-{\nu}/{\psi_k(\zt_{G(k)})}}{1-\overline{{\nu}/{\psi_k(\zt_{G(k)})}}t}}^{q(\nu,\,G(k),\,k)},\label{E:minpo_2pt_redut2}
\end{align}
where $q(\nu,\,L(k),\,k)$ and $q(\nu,\,G(k),\,k)$ are as in the statement of Theorem~\ref{T:3pt_nec}.
Now if $\eta\in\sigma(W_{L(k),\,k})$ or $\eta\in\sigma(W_{G(k),\,k})$,
then $\eta=\mu/{\psi_k(\zt_{L(k)})}$ for some $\mu\in\sigma(B_k(W_{L(k)}))$ or
$\eta=\mu/{\psi_k(\zt_{G(k)})}$ for some $\mu\in\sigma(B_k(W_{G(k)}))$, respectively, and conversely.
This observation together with \eqref{E:minpo_2pt_redut2}, \eqref{E:minpo_2pt_redut1} and 
\eqref{E:ineq_2pt_redut} establishes the first part of our theorem.
\smallskip

\noindent{\bf Case 2.} $\widetilde{G_k}(\D)\cap\partial\,\Omega_n\not=\emptyset$.

\noindent Let $\zt_0\in\D$ such that $e^{i\theta_0}\in\sigma(\widetilde{G_k}(\zt_0))$
for some $\theta_0\in\mathbb{R}$. By Lemma~\ref{L:fl1}, we have
$e^{i\theta_0}\in\sigma(\widetilde{G_k}(\zt))$ for every $\zt\in\D$. By \eqref{E:fact_aux_intp},
$e^{i\theta_0}\zt\in\sigma(\widetilde{F_k}(\zt))$. Let $\Phi\,\equiv\,F\circ\psi_k^{-1}$. Then
$\Phi\in\hol(\D,\,\Omega_n)$ and we have:
\[
e^{i\theta_0}\zt\in\sigma(B_k\circ\Phi(\zt))=B_k\{\sigma(\Phi(\zt))\}\,\,\forall\zt\in\D,
\]
where the last equality is an application of the Spectral Mapping Theorem .
For each $\zt\in\D$, let $\omega_{\zt}\in\sigma(\Phi(\zt))$ be such that
$B_{k}(\omega_{\zt})\,=\,e^{i\theta_0}\zt$. Notice that if $\zt_1\not=\zt_2$ then
$\omega_{\zt_1}\not=\omega_{\zt_2}$, whence $E\,:=\,\{\omega_{\zt}\,:\,\zt\in\D\}$ is
an uncountable set. Notice that $\omega_{\zt}$ satisfies:
\[
B_k(\omega_\zt)=e^{i\theta_0}\zt\,\,\,\text{and}\,\,\,
\mathrm{det}\left(\omega_{\zt}\mathbb{I}-\Phi(\zt)\right)
=0\,\,\forall\zt\in\D.
\]
This implies $\mathrm{det}\left(\omega_{\zt}\mathbb{I}-\Phi(\zt)\right)\,=\,0$ for every $\omega_{\zt}\in E$.
As $E$ is uncountable, it follows that 
$\mathrm{det}\left((\bcdot)\mathbb{I}-\Phi\circ(e^{-i\theta_0}B_k)\right)\equiv 0$. As $B_k$ maps $\D$
onto itself, it follows that 
\begin{equation}\label{E:blasholcor}
B_k^{-1}\{e^{i\theta_0}\zt\}\subset\sigma(\Phi(\zt))=\sigma\left(F\circ\psi_k^{-1}(\zt)\right)\,\,\forall\zt\in\D.
\end{equation}
Putting $\zt=\psi_k(\zt_{L(k)})$ and $\psi_k(\zt_{G(k)})$ respectively in \eqref{E:blasholcor}
we get 
$B_k^{-1}\{e^{i\theta_0}\psi_k(\zt_{L(k)})\}\subset\sigma(F(\zt_{L(k)}))
=\sigma(W_{L(k)})$ and
$B_k^{-1}\{e^{i\theta_0}\psi_k(\zt_{G(k)})\}\subset\sigma(F(\zt_{G(k)}))=\sigma(W_{G(k)})$.\qed
\smallskip

We now present our example that compares the condition given by Theorem~\ref{T:3pt_nec} with that of 
Costara and Ogle and Baribeau--Kamara as alluded to in Remark~\ref{Rmk:comparison}.

\begin{example}\label{Exam:comparison}
For each $n\geq 4$, there is a class of $3$-point data-sets for which 
Theorem~\ref{T:3pt_nec}\linebreak
\pagebreak

\noindent{implies that it cannot admit any $\hol(\D,\,\OM_n)$-interpolant, whereas
the conditions given by \cite{costara:osNPp05, ogle:thesis99} and by Result~\ref{Res:inequality_baribeauKamara}
provide no information.}
\end{example}

\noindent Let $0$, $a$ and $b$ be distinct points in $\D$. For each $n\geq 4$, we will construct a class of $3$-point
matricial data\,---\,of the form $\{0, \A, \B\}$ such that $0, \A, \B\in\OM_n$, where $\A$ and $\B$ will depend
on $n, a, b$\,---\,for which the aforementioned statement holds true. To this end, consider the matrices:
\begin{align*}
  \A&:=\sum_{j=0}^{n-1}\alpha_j N^j,\ \text{where $\alpha_j\in\C$ is such that}
         \ \alpha_j=0 \ \forall j\in\intgR{0}{n-3} \ \text{and} \ \alpha_{n-2}\neq 0, \\
  \B&:=\text{diag}[\beta_1,\dots,\beta_n]:=\ \text{the diagonal matrix with {\bf distinct} entries $\beta_i$}
\end{align*} 
such that $\beta_1=0$ and such that 
\begin{itemize}
 \item $\beta_i^2\neq\beta_j^2\;\;\forall i\neq j$,
 \item $|\,\beta_i\,|<|\,b\,|\;\;\forall i\in\intgR{1}{n}$, and
 \item $|\,\beta_i\,|^2<\mobi(a,\,b)\;\;\forall i\in\intgR{1}{n}$.
\end{itemize}
We shall see the relevance of these conditions presently. Here, $N$ is the nilpotent matrix
introduced in Lemma~\ref{L:minmo_lincomb_nilpo}.
\smallskip

Notice that by Lemma~\ref{L:minmo_lincomb_nilpo} the minimal polynomial for $\A$ is given by 
$\minpo{\A}(t)=t^2$ while its characteristic polynomial is $t^n$. As $n\geq 4$, $\A$ is not a non-derogatory
matrix. Thus, for the data $\{(0,0),(a,\A),(b,\B)\}$ the result given by Costara and Ogle cannot
be applied in this setting, and hence yields no information. 
\smallskip

Let us compute the form that the condition \eqref{E:inequality_baribeauKamara} takes for these data
by setting $(\zt_1, W_1)=(0, 0)$, $(\zt_2, W_2)=(a, \A)$ and $(\zt_3, W_3)=(b, \B)$.
Observe: $B_1(t)=t$ for every $t\in\D$. The Spectral Mapping Theorem then implies that 
$\sigma(B_1(W_j))=B_1(\sigma(W_j))=\sigma(W_j)$ for $j=2,3$. Hence we have:
\[
 \mathcal{H}^{\mobi}\left(\sigma\intf{(}{)}{B_1(W_{2})}{\psi_1(\zt_{2})}\cap\D,\;\;
  \sigma\intf{(}{)}{B_1(W_{3})}{\psi_1(\zt_{3})}\cap\D\right)=\max_{1\leq i\leq n}\intf{|}{|}{\,\beta_i\,}{\,b\,}.
\]
Since $\beta_i$'s are all distinct, $\nu=n$. So the condition \eqref{E:inequality_baribeauKamara}
turns out to be:
\begin{equation}\label{E:barbkamar_1}
 \max_{1\leq i\leq n}\intf{|}{|}{\,\beta_i\,}{\,b\,}\leq\mobi(a,\,b)^{1/n}.
\end{equation}
Permuting the roles of $(\zt_j,W_j)$, $j=1,2,3$, in Result~\ref{Res:inequality_baribeauKamara}
provides two other conditions. In the case when $(\zt_1, W_1)=(b, \B)$ the condition \eqref{E:inequality_baribeauKamara}
holds trivially, because its left-hand side reduces to $0$. On the other hand if $(\zt_1, W_1)=(a, \A)$,
then under the restriction $\beta_i^2\neq\beta_j^2$, for $i\neq j$, and that $|\,\beta_i\,|^2<\mobi(a,\,b)$ for
every $i\in\intgR{1}{n}$, we leave it to the reader to check that this condition turns out to be:
\begin{equation}\label{E:barbkamar_2}
 \max_{1\leq i\leq n}\intf{|}{|}{\,\beta_i^2\,}{\,\psi_a(b)\,}\leq |\,b\,|^{1/n}.
\end{equation} 

Now let us compute the form that the necessary condition provided by Theorem~\ref{T:3pt_nec} takes for the above data
in the case $k=1$, $L(k)=2$, $G(k)=3$. To this end, we observe first:
\[
 \sub{\mu\in\sigma\left(B_1(W_{2})\right)}{\max}\prod_{\nu\in\sigma\left(B_1(W_{3})\right)}
 {\mathcal{M}_{\D}\left(\intf{}{}{\mu}{\psi_1(\zt_{2})},\,\intf{}{}{\nu}{\psi_1(\zt_{3})}\right)}^{q(\nu,\,3,\,1)}
  = \prod_{i=1}^n{\intf{|}{|}{\,\beta_i\,}{\,b\,}}^{q(\beta_i,\,3,\,1)},
\]
which is equal to zero since $q(\beta_i,\,3,\,1)=1$ for every $i\in\intgR{1}{n}$ and $\beta_1=0$.
On the other hand:
\[
 \sub{\mu\in\sigma\left(B_1(W_{3})\right)}{\max}\prod_{\nu\in\sigma\left(B_1(W_{2})\right)}
 {\mathcal{M}_{\D}\left(\intf{}{}{\mu}{\psi_1(\zt_{3})},\,\intf{}{}{\nu}{\psi_1(\zt_{2})}\right)}^{q(\nu,\,2,\,1)}
 =\max\left\{{\intf{|}{|}{\,\beta_i\,}{\,b\,}}^{q(0,\,2,\,1)}:i\in\intgR{1}{n}\right\}.
\] 
Notice $m(2,\,0)$\,---\,the multiplicity of $0$ as a zero of $\minpo{\A}$\,---\,is $2$ and $\mathsf{ord}_{0}{B'_1}=0$ whence 
$q(0,\,2,\,1)=2$. Hence, one of the conditions given by Theorem~\ref{T:3pt_nec} is:
\begin{equation}\label{E:barbkamar_3}
 \max_{1\leq i\leq n}{\intf{|}{|}{\,\beta_i\,}{\,b\,}}^2\leq\mobi(a,\,b).
\end{equation} 
Notice that 
\begin{equation}\label{E:barbkamar1_2_geq_3}
 |\,b\,|\mobi(a,\,b)^{1/2}<\min\left\{|\,b\,|{\mobi(a,\,b)}^{1/n},\, |\,b\,|^{1/2n}{\mobi(a,\,b)}^{1/2}\right\}\;\;\forall n\geq 3.
\end{equation}
Now, let us choose $\{\beta_1,\dots,\beta_n\}\subset\D$ such that, in addition to the conditions listed above,
\[
 |\,\beta_i\,|\leq \min\left\{|\,b\,|{\mobi(a,\,b)}^{1/n},\, |\,b\,|^{1/2n}{\mobi(a,\,b)}^{1/2}\right\}\;\;\forall i\in\intgR{2}{n},
\]
and such that for some $i_0\in\intgR{2}{n}$,
\[
 |\,\beta_{i_0}\,|>|\,b\,|\mobi(a,\,b)^{1/2}.
\]
This is possible owing to \eqref{E:barbkamar1_2_geq_3}. We now see that all forms of the condition 
arising from Result~\ref{Res:inequality_baribeauKamara} are satisfied by the given data-set, while 
\eqref{E:barbkamar_3} does not hold true. Hence, Theorem~\ref{T:3pt_nec} implies that there does not exist
an $F\in\hol(\D,\,\OM_n)$ such that $F(0)=0$, $F(a)=\A$ and $F(b)=\B$ while Result~\ref{Res:inequality_baribeauKamara}
provides no information.
\medskip

\section{A Schwarz lemma for holomorphic correspondences}\label{S:SchLem_holcorres}

This section is dedicated to the proof of Theorem~\ref{T:Schwarzlemma_corres_corol}. However, as hinted
at in Section~\ref{S:intro}, there is a more precise inequality, from which Theorem~\ref{T:Schwarzlemma_corres_corol}
follows. We begin, therefore, with the following:

\begin{theorem}\label{T:SchwLem_holcorres}
Let $\Omega$ be a bounded domain in $\C$ and let $\Gamma$ be a proper holomorphic correspondence
from $\D$ to $\Omega$. Then, for every $\zt_1,\zt_2\in\D$ we have
\[
\max\Big\{\sub{\mu\in\bv{\multF(\zt_2)}}{\max}\prod_{\nu\in\bv{\multF(\zt_1)}}C_\Omega(\nu,\,\mu), \
\sub{\mu\in\bv{\multF(\zt_1)}}{\max}\prod_{\nu\in\bv{\multF(\zt_2)}}C_\Omega(\nu,\,\mu)\Big\}
\leq\mathcal{M}_{\D}(\zt_1,\,\zt_2),
\]
where $\bv{\multF(\bcdot)}$ is as in Section~\ref{S:intro}, read together
with Remark~\ref{Rem:not_mltply}.
\end{theorem}

We remind the reader that $C_{\Omega}$ here is as defined by \eqref{E:alt_def_cara}. 
The proof of the above theorem is an easy consequence of the following lemma.

\begin{lemma}\label{L:flsch_hol_corres}
Let $\Gamma$ be as in the statement of Theorem~\ref{T:SchwLem_holcorres}. Then
\[
\sub{\mu\in\bv{\multF(\zt)}}{\max}\prod_{p\in\bv{\multF(0)}}C_{\Omega}(p,\,\mu)\leq |\,\zt\,| \ \forall\zt\in\D.
\]
\end{lemma}

\begin{proof}
By Lemma~\ref{L:corres_equa_zeroset}, there exists a positive integer $n$ and
functions $a_j\in\hol(\D)$, $j=1,\dots,n$, such that
\[
 \Gamma=\Big\{(z,w)\in\D\times\Omega\,:\,
w^n+\sum_{j=1}^n(-1)^j a_j(z)w^{n-j}=0\Big\}.
\]
Note that by our definition of $\bv{\multF(\bcdot)}$
\begin{align*}
\bv{\multF(\bcdot)}=& \ \text{the {\bf list}, repeated according to multiplicity, of the zeros of}\\
& w^n+\sum_{j=1}^n(-1)^j a_j(\bcdot)w^{n-j}.
\end{align*}
Let us now define an open set in $M_{n}(\C):S_{n}(\Omega):=\{A\in M_{n}(\C):\sigma(A)\subset \Omega\}$. 
Let $\Phi\in\hol(\D,\,M_{n}(\C))$ that is defined by
\[
\Phi(z):=\text{companion matrix corresponding to the polynomial}  \ w^n+\sum_{j=1}^n(-1)^j a_j(z)w^{n-j}.
\]
In the notation introduced by Remark~\ref{Rem:not_mltply}, $\bv{\sigma(A)}$ will denote the list of eigenvalues of $A$
repeated according to their multiplicity. In this notation, we have $\bv{\sigma(\Phi(z))}=\bv{\multF(z)}\subset\Omega$. 
Hence $\Phi(\D)\subset S_n(\Omega)$. Now we choose an arbitrary $z\in\Omega$ and fix it. Consider $B\in\hol(\Omega)$
defined by
\[
B:=\prod_{p\in\bv{\multF(0)}}G_{\Omega}(p,\,z;\,\bcdot),
\]
where $G_{\Omega}(p,\,z;\,\bcdot)$ denotes the Carath\'{e}odory extremal for points $p,z\in\Omega$,
whose existence was discussed in Section~\ref{S:nott_comx_geom}. As $B$ is holomorphic in $\Omega$,
it induces\,---\,via the holomorphic functional calculus\,---\, a map (which we continue to denote by $B$)
from $S_{n}(\Omega)$ to $M_n(\C)$. The Spectral Mapping Theorem tells us that $\sigma(B(A))=B(\sigma(A))\subset \D$
for every $A\in S_{n}(\Omega)$. Hence $B(A)\subset\Omega_n$ for every $A\in S_{n}(\Omega)$.
\smallskip

\noindent{\bf Claim.} $B(\Phi(0))=0$.

\noindent To see this we write:
\[
B(w)=\big(\prod\nolimits_{p\in\bv{\multF(0)}}(w-p)\big)g(w), \ w\in\Omega,
\]
where $g\in\hol(\Omega)$. Hence, since\,---\,by the holomorphic functional calculus\,---\,the assignment 
$f\longmapsto f(\Phi(0)), \ f\in\hol(\Omega)$, is multiplicative, as discussed in Section~\ref{S:holo_fc}
(also see Remark~\ref{Rem:matpara_func}), we get 
\[
B(\Phi(0))=\big(\prod\nolimits_{p\in\bv{\multF(0)}}(\Phi(0)-p\,\mathbb{I})\big)g(\Phi(0)).
\]
As $\bv{\multF(0)}=\bv{\sigma(\Phi(0))}$, Cayley--Hamilton Theorem implies that the product term in the 
right hand side of the above equation is zero, whence the claim.
\smallskip

Consider the map $\Psi$ defined by:
\[
\Psi(\zt):=B\circ\Phi(\zt), \ \zt\in\D.
\]
It is a fact that $\Psi\in\hol(\D)$. Moreover from the above claim and the discussion just before it,
we have $\Psi\in\hol(\D,\,\Omega_n)$ with $\Psi(0)=0$. By Lemma~\ref{L:fl2}
there exists $\widetilde{\Psi}\in\hol(\D,\,\overline{\Omega}_{n})$ such that 
\begin{equation}\label{E:fact_Psi}
\Psi(\zt)=\zt\widetilde{\Psi}(\zt) \ \forall\zt\in\D.
\end{equation}
From the definition of $\Psi$, \eqref{E:fact_Psi}, and from the Spectral Mapping Theorem, we get 
\begin{align*}
B(\bv{\sigma(\Phi(\zt))})&=\zt\bv{\sigma(\widetilde{\Psi}(\zt))}\\
\implies |\,B(\mu)\,|&\leq |\,\zt\,| \ \forall\mu\in\bv{\sigma(\Phi(\zt))}=\bv{\multF(\zt)}, \ \forall\zt\in\D.
\end{align*}
This in turn implies that 
\[
\prod_{p\in\bv{\multF(0)}}|\,G_{\Omega}(p,\,z;\,\mu)\,| \leq |\,\zt\,| \ \forall\zt\in\D, \ \forall\mu\in\bv{\multF(\zt)}.
\]
Since $z$ is arbitrary we can take $z=\mu$ for some $\mu\in\bv{\multF(\zt)}$. This with the observation that 
$G_{\Omega}(p,\,\mu;\,\mu)=C_{\Omega}(p,\,\mu)$ establishes the lemma.
\end{proof}
We are now ready to give:
\begin{proof}[The proof of the Theorem~\ref{T:SchwLem_holcorres}]
Fix $\zt_1$, $\zt_2\in\D$ and consider the correspondence $\widetilde{\Gamma}$ such that
\[
\bv{{F_{\widetilde{\Gamma}}(\zt)}}=\bv{{\multF(\psi_{\zt_1}^{-1}(\zt))}}.
\]
It is easy to see that $\widetilde{\Gamma}$ is a proper holomorphic correspondence from $\D$ to $\Omega$.
So from the lemma above and with the observation that
$\bv{{F_{\widetilde{\Gamma}}(0)}}=\bv{{\multF(\zt_1)}}$ we get 
\begin{equation}\label{E:ineq_1}
\sub{\mu\in\bv{{F_{\widetilde{\Gamma}}(\zt)}}}{\max}\prod_{p\in\bv{{\multF(\zt_1)}}}C_{\Omega}(p,\,\mu)
\leq |\,\zt\,| \ \forall\zt\in\D.
\end{equation}
Putting $\zt=\psi_{\zt_1}(\zt_2)$ in \eqref{E:ineq_1} gives us
\begin{equation}\label{E:ineq_2}
\sub{\mu\in\bv{\multF(\zt_2)}}{\max}\prod_{p\in\bv{{\multF(\zt_1)}}}C_{\Omega}(p,\,\mu)
\leq \mathcal{M}_{\D}(\zt_1,\,\zt_2).
\end{equation}
Interchanging the role of $\zt_1$ and $\zt_2$ in the above discussion, we get
\begin{equation}\label{E:ineq_3}
\sub{\mu\in\bv{{\multF(\zt_1)}}}{\max}\prod_{p\in\bv{{\multF(\zt_2)}}}C_{\Omega}(p,\,\mu)
\leq \mathcal{M}_{\D}(\zt_1,\,\zt_2).
\end{equation}
From \eqref{E:ineq_3} and \eqref{E:ineq_2} the result follows. 
\end{proof}

Theorem~\ref{T:Schwarzlemma_corres_corol} is a corollary of Theorem~\ref{T:SchwLem_holcorres}.
This is almost immediate; we probably just require a few words about the Hausdorff distance induced by $C_{\Omega}$.
We refer the reader to \cite[p.~279]{munk:topo_74} for the definition of the Hausdorff distance\,---\,which is a distance on the set of non-empty
closed, bounded subsets of a distance space $(X,\,d)$. In our case $(X,\,d)=(\Omega,\,C_{\Omega})$ and it is
easy to check that
\[
\hausd\big(\multF(\zt_1),\,\multF(\zt_2)\big)=\max\Big\{\max_{w\in \multF(\zt_1)}
{\rm dist}_{C_{\Omega}}\big(w,\,\multF(\zt_2)\big),\,
\max_{w\in \multF(\zt_2)}{\rm dist}_{C_{\Omega}}\big(w,\,\multF(\zt_1)\big)\Big\},
\]
given $\zt_1$, $\zt_2\in\D$, and where, given $p\in\Omega$, $\emptyset\not= S\subset\Omega$,
${\rm dist}_{C_{\Omega}}(p,\, S):=\inf\nolimits_{q\in S}C_{\Omega}(p,\,q)$.
Clearly, there exists $j,\,k\in\{1,2\}$ with $j\not=k$ such that
\[
\hausd\big(\multF(\zt_1),\,\multF(\zt_2)\big)^n\leq
\sub{\mu\in\bv{\multF(\zt_j)}}{\max}\prod_{\nu\in\bv{\multF(\zt_k)}}C_\Omega(\nu,\,\mu),
\]
where $n$ is the multiplicity of $\Gamma$. Combining the above inequality
with the inequality in Theorem~\ref{T:SchwLem_holcorres}, we establish
Theorem~\ref{T:Schwarzlemma_corres_corol}.
\medskip

\section*{Acknowledgements}
I wish to thank my thesis adviser Gautam Bharali for the many helpful discussions during the course
of this work. I am especially grateful for his supporting me as a Research Fellow under his Swarnajayanti
Fellowship (Grant No.~DST/SJF/MSA-02/2013-14).

\end{document}